\documentclass[12pt,a4paper]{smfart}
\usepackage{smfenum}
\usepackage{amsmath,amssymb,amsthm,latexsym,MnSymbol}
\usepackage[utf8]{inputenc}
\usepackage{url}
\usepackage[pdftex]{graphicx}
\usepackage{tikz,tikz-cd}
\usepackage[pdftex]{hyperref}
\usepackage[all]{xy}
\usepackage[margin=2cm]{geometry}

\newcommand{\Z}{\mathbf{Z}}
\newcommand{\Q}{\mathbf{Q}}
\newcommand{\Qb}{\overline{\Q}}

\newcommand{\C}{\mathbf{C}}

\newcommand{\PP}{\mathbf{P}}

\DeclareMathOperator{\dv}{div}

\DeclareMathOperator{\Gal}{Gal}
\DeclareMathOperator{\GL}{GL}

\DeclareMathOperator{\SL}{SL}

\newtheorem{thm}{Theorem}
\newtheorem{lem}[thm]{Lemma}
\newtheorem{pro}[thm]{Proposition}

\theoremstyle{definition}
\newtheorem{definition}[thm]{Definition}

\theoremstyle{remark}

\newtheorem{remarks}[thm]{Remarks}
\newtheorem{example}[thm]{Example}

\newtheorem{question}[thm]{Question}

\author{François Brunault}
\email{francois.brunault@ens-lyon.fr}
\address{ÉNS Lyon, UMPA, 46 allée d'Italie, 69007 Lyon, France}

\begin{document}

\renewcommand{\refname}{References}

\mainmatter

\begin{center}
\Large
\textbf{Parametrizing elliptic curves by modular units}
\normalsize
\end{center}

\vspace{.5cm}

\emph{Abstract.}
It is well-known that every elliptic curve over the rationals admits a parametrization by means of modular functions. In this short note, we show that only finitely many elliptic curves over $\Q$ can be parametrized by modular units. This answers a question raised by Zudilin in a recent work on Mahler measures. Further, we give the list of all elliptic curves $E$ of conductor up to $1000$ parametrized by modular units supported in the rational torsion subgroup of $E$. Finally, we raise several open questions.

\vspace{.5cm}

Since the work of Boyd \cite{boyd:expmath}, Deninger \cite{deninger:mahler} and others, it is known that there is a close relationship between Mahler measures of polynomials and special values of $L$-functions. Although this relationship is still largely open, some strategies have been identified in several instances. Specifically, let $P \in \Q[x,y]$ be a polynomial whose zero locus defines an elliptic curve $E$. If the polynomial $P$ is tempered, then the Mahler measure of $P$ can be expressed in terms of a regulator integral
\begin{equation}\label{reg integral}
\int_\gamma \log |x| d\arg(y)-\log |y| d\arg(x)
\end{equation}
where $\gamma$ is a (non necessarily closed) path on $E$ (see \cite{deninger:mahler,zudilin}). If the curve $E$ happens to have a parametrization by \emph{modular units} $x(\tau)$, $y(\tau)$, then we may change to the variable $\tau$ in (\ref{reg integral}) and try to compute the regulator integral using \cite[Thm 1]{zudilin}. In favourable cases, this leads to an identity between the Mahler measure of $P$ and $L(E,2)$: see for example \cite[\S 3]{zudilin} and the references therein. The following natural question, raised by Zudilin, thus arises:

\begin{center}
\emph{Which elliptic curves can be parametrized by modular units?}
\end{center}

We show in Section \ref{finiteness} that only finitely many elliptic curves over $\Q$ can be parametrized by modular units. The proof uses Watkins' lower bound on the modular degree of elliptic curves. Further, we give in Section \ref{preimages} the list of all elliptic curves $E$ of conductor up to $1000$ parametrized by modular units supported in the rational torsion subgroup of $E$. It turns out that there are 30 such elliptic curves. Finally, we raise in Section \ref{questions} several open questions.

\section{A finiteness result}\label{finiteness}

\begin{definition}\label{def param modunits}
Let $E/\Q$ be an elliptic curve of conductor $N$. We say that $E$ can be \emph{parametrized by modular units} if there exist two modular units $u,v \in \mathcal{O}(Y_1(N))^\times$ such that the function field $\Q(E)$ is isomorphic to $\Q(u,v)$.
\end{definition}

\begin{thm}\label{thm finiteness}
There are only finitely many elliptic curves over $\Q$ which can be parametrized by modular units.
\end{thm}

Let $E/\Q$ be an elliptic curve of conductor $N$. Assume that $E$ can be parametrized by two modular units $u,v$ on $Y_1(N)$. Then there is a finite morphism $\varphi : X_1(N) \to E$ and two rational functions $f,g \in \Q(E)^\times$ such that $\varphi^*(f)=u$ and $\varphi^*(g)=v$.

Let $E_1$ be the $X_1(N)$-optimal elliptic curve in the isogeny class of $E$, and let $\varphi_1 : X_1(N) \to E_1$ be an optimal parametrization. By \cite[Prop 1.4]{stevens}, there exists an isogeny $\lambda : E_1 \to E$ such that $\varphi = \lambda \circ \varphi_1$. Consider the functions $f_1=\lambda^*(f)$ and $g_1=\lambda^*(g)$. Note that $u=\varphi_1^*(f_1)$ and $v=\varphi_1^*(g_1)$. Theorem \ref{thm finiteness} is now a consequence of the following result.

\begin{thm}\label{thm 2}
If $N$ is sufficiently large, then $\varphi_1^*(\Q(E_1)) \cap \mathcal{O}(Y_1(N))=\Q$.
\end{thm}

\begin{proof}[Proof]
Let $C_1(N)$ be the set of cusps of $X_1(N)$. Let $f \in \Q(E_1) \backslash \Q$ be such that $\varphi_1^*(f) \in \mathcal{O}(Y_1(N))$. Let $P$ be a pole of $f$. Then $\varphi_1^{-1}(P)$ must be contained in $C_1(N)$, and we have
\begin{equation*}
\deg \varphi_1 = \sum_{Q \in \varphi_1^{-1}(P)} e_{\varphi_1}(Q) \leq \sum_{Q \in C_1(N)} e_{\varphi_1}(Q).
\end{equation*}
Let $g_N$ be the genus of $X_1(N)$. By the Riemann-Hurwitz formula for $\varphi_1$, we have
\begin{equation*}
2g_N-2 = \sum_{Q \in X_1(N)} (e_{\varphi_1}(Q)-1).
\end{equation*}
It follows that
\begin{align*}
\deg \varphi_1 & \leq \#C_1(N) + \sum_{Q \in C_1(N)} (e_{\varphi_1}(Q)-1)\\
& \leq \#C_1(N) + 2g_N-2.
\end{align*}
By the classical genus formula \cite[Prop 1.40]{shimura:book}, and since $X_1(N)$ has no elliptic points for $N \geq 4$, we have
\begin{equation*}
\#C_1(N)+2g_N-2 = \frac{1}{12} [\SL_2(\Z):\Gamma_1(N)] = \frac{\phi(N) \nu(N)}{12} \qquad (N \geq 4)
\end{equation*}
where $\phi(N)$ denotes Euler's function, and $\nu(N)$ is defined by
\begin{equation*}
\nu(N) = N \prod_{i=1}^k (1+\frac{1}{p_i}) \qquad \textrm{if } N = \prod_{i=1}^k p_i^{\alpha_i}.
\end{equation*}
We thus get
\begin{equation}\label{majoration deg phi1}
\deg \varphi_1 \leq \frac{\phi(N) \nu(N)}{12}.
\end{equation}

We are now going to show that (\ref{majoration deg phi1}) contradicts lower bounds of Watkins on the modular degree if $N$ is sufficiently large. Let $E_0$ be the strong Weil curve in the isogeny class of $E$. We have a commutative diagram
\begin{equation}\label{E1E0}
\begin{tikzcd}
X_1(N) \arrow{d}{\varphi_1} \arrow{r}{\pi} & X_0(N) \arrow{d}{\varphi_0}\\
E_1 \arrow{r}{\lambda_0} & E_0.
\end{tikzcd}
\end{equation}
We deduce that
\begin{equation*}
\deg \varphi_1 = \frac{\deg \pi \cdot \deg \varphi_0}{\deg \lambda_0}.
\end{equation*}
We have $\deg \pi = \frac{\phi(N)}{2}$. For every $\alpha \in (\Z/N\Z)^\times/\pm 1$, there exists a unique point $A(\alpha) \in E_1(\Q)_{\mathrm{tors}}$ such that $\varphi_1 \circ \langle \alpha \rangle = t_{A(\alpha)} \circ \varphi_1$, where $t_{A(\alpha)}$ denotes translation by $A(\alpha)$. The map $\alpha \mapsto A(\alpha)$ is a morphism of groups and its image is $\ker(\lambda_0)$. It follows that $\deg(\lambda_0) \leq \# E_1(\Q)_{\mathrm{tors}} \leq 16$. By \cite{watkins}, we have $\deg \varphi_0 \gg N^{7/6-\varepsilon}$ for any $\varepsilon>0$. It follows that $\deg \varphi_1 \gg \phi(N) N^{7/6-\varepsilon}$. Since $\nu(N) \ll N^{1+\varepsilon}$ for any $\varepsilon>0$, this contradicts (\ref{majoration deg phi1}) for $N$ sufficiently large.
\end{proof}

It would be interesting to determine the complete list of elliptic curves over $\Q$ parametrized by modular units. Unfortunately, the bound provided by Watkins' result, though effective, is too large to permit an exhaustive search.

\section{Preimages of torsion points under modular parametrizations}\label{preimages}

In order to find elliptic curves parametrized by modular units, we consider the following related problem. Let $E$ be an elliptic curve over $\Q$ of conductor $N$, and let $\varphi : X_1(N) \to E$ be a modular parametrization sending the $0$-cusp to $0$. By the Manin-Drinfeld theorem, the image by $\varphi$ of a cusp of $X_1(N)$ is a torsion point of $E$. Conversely, given a point $P \in E_{\mathrm{tors}}$, when does the preimage of $P$ under $\varphi$ consist only of cusps? The link between this question and parametrizations by modular units is given by the following easy lemma.

\begin{lem}\label{lem preimage}
Suppose that there exists a subset $S$ of $E(\Q)_{\mathrm{tors}}$ satisfying the following two conditions:
\begin{enumerate}
\item We have $\varphi^{-1}(S) \subset C_1(N)$.
\item There exist two functions $f,g$ on $E$ supported in $S$ such that $\Q(E)=\Q(f,g)$.
\end{enumerate}
Then $E$ can be parametrized by modular units.
\end{lem}

\begin{proof}[Proof]
By condition (1), the functions $u=\varphi^*(f)$ and $v=\varphi^*(g)$ are modular units of level $N$, and by condition (2), we have $\Q(E) \cong \Q(u,v)$.
\end{proof}

We are therefore led to search for elliptic curves $E/\Q$ admitting sufficiently many torsion points $P$ such that $\varphi^{-1}(P) \subset C_1(N)$.

We first give an equivalent form of condition (2) in Lemma \ref{lem preimage}.

\begin{pro}\label{pro FS}
Let $S$ be a subset of $E(\Q)_{\textrm{tors}}$. Let $\mathcal{F}_S$ be the set of nonzero functions $f$ on $E$ which are supported in $S$. The following conditions are equivalent:
\begin{enumerate}
\item[(a)] \label{pro FS a} There exist two functions $f,g \in \mathcal{F}_S$ such that $\Q(E)=\Q(f,g)$.
\item[(b)] \label{pro FS b} The field $\Q(E)$ is generated by $\mathcal{F}_S$.
\item[(c)] \label{pro FS c} We have $\# S \geq 3$, and there exist two points $P,Q \in S$ such that $P-Q$ has order $\geq 3$.
\end{enumerate}
\end{pro}

In order to prove Proposition \ref{pro FS}, we show the following lemma.

\begin{lem}\label{lem FS}
Let $P \in E(\Q)_{\textrm{tors}}$ be a point of order $n \geq 2$. Let $f_P$ be a function on $E$ such that $\dv(f_P)=n(P)-n(0)$. Then the extension $\Q(E)/\Q(f_P)$ has no intermediate subfields. Moreover, if $P,P' \in E(\Q)_{\textrm{tors}}$ are points of order $n \geq 4$ such that $\Q(f_P)=\Q(f_{P'})$, then $P=P'$.
\end{lem}

\begin{proof}[Proof]
Let $K$ be a field such that $\Q(f_P) \subset K \subset \Q(E)$. If $K$ has genus $1$, then $K$ is the function field of an elliptic curve $E'/\Q$ and $f_P$ factors through an isogeny $\lambda : E \to E'$. Then $\dv(f_P)$ must be invariant under translation by $\ker(\lambda)$. This obviously implies $\ker(\lambda)=0$, hence $K=\Q(E)$. If $K$ has genus $0$, then we have $K=\Q(h)$ for some function $h$ on $E$, and we may factor $f_P$ as $g \circ h$ with $g : \PP^1 \to \PP^1$. We may assume $h(P)=0$ and $h(0)=\infty$. Then $g^{-1}(0) = \{0\}$ and $g^{-1}(\infty)=\{\infty\}$, which implies $g(t)=at^m$ for some $a \in \Q^\times$ and $m \geq 1$. Thus $\dv(f)=m\dv(h)$. Since $\dv(h)$ must be a principal divisor, it follows that $m=1$ and $K=\Q(f_P)$.

Let $P, P' \in E(\Q)$ be points of order $n \geq 4$ such that $\Q(f_P)=\Q(f_{P'})$ and $P \neq P'$. Then $f_{P'} = (af_P+b)/(cf_P+d)$ for some $\begin{pmatrix} a & b \\ c & d \end{pmatrix} \in \GL_2(\Q)$. Considering the divisors of $f_P$ and $f_{P'}$, we must have $f_{P'}=af_P+b$ for some $a,b \in \Q^\times$. Then the ramification indices of $f_P : E \to \PP^1$ at $P$, $P'$, $0$ are equal to $n$, which contradicts the Riemann-Hurwitz formula for $f_P$.
\end{proof}

\begin{proof}[Proof of Proposition \ref{pro FS}]
It is clear that $(\mathrm{a})$ implies $(\mathrm{b})$. Let us show that $(\mathrm{b})$ implies $(\mathrm{c})$. If $\# S \leq 2$, then $\mathcal{F}_S/\Q^\times$ has rank at most $1$ and cannot generate $\Q(E)$. Assume that for every points $P,Q \in S$, we have $P-Q \in E[2]$. Translating $S$ if necessary, we may assume $0 \in S$. It follows that $S \subset E[2]$ and $\mathcal{F}_S \subset \Q(x) \subsetneq \Q(E)$.

Finally, let us assume $(\mathrm{c})$. Translating $S$ if necessary, we may assume $0 \in S$. Let us first assume that $S$ contains a point $P$ of order $2$. Then $\Q(f_P) = \Q(x)$ has index $2$ in $\Q(E)$ and is the fixed field with respect to the involution $\sigma : p \mapsto -p$ on $E$. By assumption, there exist two points $Q, R \in S$ such that $Q-R$ has order $n \geq 3$. Let $g$ be a function on $E$ such that $\dv(g)=n(Q)-n(R)$. Then it is easy to see that $\dv(g)$ is not invariant under $\sigma$. It follows that $g \not\in \Q(f_P)$ and $\Q(f_P,g)=\Q(E)$. Let us now assume that $S \cap E[2]=\{0\}$. By assumption, $S$ contains two distinct points $P,Q$ having order $\geq 3$. If $P$ or $Q$ has order $\geq 4$, then Lemma \ref{lem FS} implies that $\Q(f_P,f_Q)=\Q(E)$. If $P$ and $Q$ have order $3$, then we must have $Q=-P$ because $\Q(E[3])$ contains $\Q(\zeta_3)$. It follows that the function $g$ on $E$ defined by $\dv(g)=(P)+(-P)-2(0)$ has degree $2$, so we have $g \not\in \Q(f_P)$ and $\Q(f_P,g)=\Q(E)$.
\end{proof}

Let $E/\Q$ be an elliptic curve of conductor $N$. Fix a Néron differential $\omega_E$ on $E$, and let $f_E$ be the newform of weight $2$ and level $N$ associated to $E$. We define $\omega_{f_E}=2\pi i f_E(z)dz$. Let $\varphi_E : X_1(N) \to E$ be a modular parametrization of minimal degree. We have $\varphi_E^* \omega_E = c_E \omega_{f_E}$ for some integer $c_E \in \Z-\{0\}$ \cite[Thm 1.6]{stevens}, and we normalize $\varphi_E$ so that $c_E>0$. Conjecturally, we have $c_E=1$ \cite[Conj. I]{stevens}.

We now describe an algorithm to compute the set $S_E$ of points $P \in E(\Q)_{\mathrm{tors}}$ such that $\varphi_E^{-1}(P) \subset C_1(N)$. Let $P \in E(\Q)_{\mathrm{tors}}$. We define an integer $e_P$ by
\begin{equation*}
e_P = \sum_{\substack{x \in C_1(N) \\ \varphi_E(x)=P}} e_{\varphi_E}(x).
\end{equation*}
It is clear that $\varphi_E^{-1}(P) \subset C_1(N)$ if and only if $e_P=\deg \varphi_E$. Let $d$ be a divisor of $N$, and let $C_d$ be the set of cusps of $X_1(N)$ of denominator $d$ (that is, the set of cusps $\frac{a}{b}$ satisfying $(b,N)=d$). Every cusp $x \in C_d$ can be written (non uniquely) as $x=\langle \alpha \rangle \sigma(\frac1d)$ with $\alpha \in (\Z/N\Z)^\times/\pm 1$ and $\sigma \in \Gal(\Q(\zeta_d)/\Q)$. Since $e_{\varphi_E}(x)=e_{\varphi_1}(x)=e_{\varphi_1}(1/d)$, we get
\begin{equation*}
e_P = \sum_{d |N} e_{\varphi_1}(1/d) \cdot \# \{x \in C_d : \varphi_E(x)=P\}.
\end{equation*}
Recall that for each $\alpha \in (\Z/N\Z)^\times$, there exists a unique point $A(\alpha) \in E(\Q)_{\mathrm{tors}}$ such that $\varphi_E \circ \langle \alpha \rangle = t_{A(\alpha)} \circ \varphi_E$, where $t_{A(\alpha)}$ denotes translation by $A(\alpha)$. We let $A_E \subset E(\Q)_{\mathrm{tors}}$ be the image of the map $\alpha \mapsto A(\alpha)$. Note that the set $\{x \in C_d : \varphi_E(x)=P\}$ is empty unless $\varphi_E(1/d) \in P+A_E$, in which case we have $\varphi_E(C_d) =P+A_E$ and the number of cusps $x \in C_d$ such that $\varphi_E(x)=P$ is given by $\# C_d / \# A_E$. Thus we get
\begin{equation*}
e_P = \frac{1}{\# A_E} \sum_{\substack{d |N \\ \varphi_E(1/d) \in P+A_E}} e_{\varphi_1}(1/d) \cdot \# C_d.
\end{equation*}
Furthermore, let $\pi : X_1(N) \to X_0(N)$ and $\varphi_0 : X_0(N) \to E_0$ be the maps as in (\ref{E1E0}). The ramification index of $\pi$ at $\frac{1}{d}$ is equal to $(d,N/d)$. Thus $e_{\varphi_1}(1/d)=(d,N/d) \cdot e_{\varphi_0}(1/d)$. The quantity $e_{\varphi_0}(1/d)$ is equal to the order of vanishing of $\omega_{f_E}$ at the cusp $1/d$, and may be computed numerically (see \cite[\S 7]{brunault:ephi}). Moreover, the number of cusps of $X_0(N)$ of denominator $d$ is given by $\phi((d,N/d))$. It follows that $\# C_d = \phi((d,N/d)) \cdot \phi(N)/(2 (d,N/d))$ and we get
\begin{equation}\label{formula eP}
e_P = \frac{\phi(N)}{2\# A_E} \sum_{\substack{d |N \\ \varphi_E(1/d) \in P+A_E}} e_{\varphi_0}(1/d) \cdot \phi((d,N/d)).
\end{equation}
Finally, using notations from Section \ref{finiteness}, the modular degree of $E$ may be computed as
\begin{equation}\label{formula degphiE}
\deg \varphi_E = \frac{\phi(N)}{2} \cdot \frac{\operatorname{covol}(\Lambda_{E_0})}{\operatorname{covol}(\Lambda_E)} \cdot \deg \varphi_0
\end{equation}
where $\Lambda_{E_0}$ and $\Lambda_E$ denote the Néron lattices of $E_0$ and $E$. We read off the modular degree $\deg \varphi_0$ from Cremona's tables \cite[Table 5]{cremona:tables}. Formulas (\ref{formula eP}) and (\ref{formula degphiE}) lead to the following algorithm.

\begin{enumerate}
\item Compute generators $\alpha_1,\ldots,\alpha_r$ of $(\Z/N\Z)^\times$.
\item For each $j$, compute numerically $\int_{z_0}^{\langle \alpha_j \rangle z_0} \omega_{f_E}$ for $z_0=(-\alpha_j+i)/N$.
\item Deduce $A_j = A(\alpha_j) \in E(\Q)_{\mathrm{tors}}$.
\item Compute the subgroup $A_E$ generated by $A_1,\ldots,A_r$.
\item Compute the list $(P_1,\ldots,P_n)$ of all rational torsion points on $E$.
\item Initialize a list $(e_{P_1},\ldots,e_{P_n})=(0,\ldots,0)$.
\item For each $d$ dividing $N$, do the following:
\begin{enumerate}
\item Compute numerically $z_d = \int_0^{1/d} \omega_{f_E}$.
\item Check whether the point $Q_d = \varphi_E(1/d)$ is rational or not.
\item If $Q_d$ is rational, then do the following:
\begin{enumerate}
\item Compute numerically $e_{\varphi_0}(1/d)$.
\item For each $B \in A_E$, do $e_{Q_d+B} \leftarrow e_{Q_d+B} + e_{\varphi_0}(1/d) \phi((d,N/d))$.
\end{enumerate}
\end{enumerate}
\item Output $S_E = \{P \in E(\Q)_{\mathrm{tors}} : e_P = \# A_E \cdot \frac{\operatorname{covol}(\Lambda_{E_0})}{\operatorname{covol}(\Lambda_E)} \cdot \deg \varphi_0\}$.
\end{enumerate}

The following table gives all elliptic curves $E$ of conductor $\leq 1000$ such that $S_E$ satisfies condition $(\mathrm{c})$ of Proposition \ref{pro FS}. Computations were done using Pari/GP \cite{pari273} and the Modular Symbols package of Magma \cite{magma}.

\begin{table}[h!]
\begin{tabular}{c|c|c||}
$E$ & $E(\Q)_{\mathrm{tors}}$ & $S_E$ \\
\hline
$11a3$ & $\Z/5\Z$ & $E(\Q)_{\mathrm{tors}}$ \\
$14a1$ & $\Z/6\Z$ & $\{0, (9, 23), (1, -1), (2, -5)\}$  \\
$14a4$ & $\Z/6\Z$ & $E(\Q)_{\mathrm{tors}}$  \\
$14a6$ & $\Z/6\Z$ & $\{0, (2, -2), (2, -1)\}$  \\
$15a1$ & $\Z/4\Z \times \Z/2\Z$ & $\{0, (-2, 3), (-1, 0), (8, 18)\}$  \\
$15a3$ & $\Z/4\Z \times \Z/2\Z$ & $\{0, (0, 1), (1, -1), (0, -2)\}$  \\
$15a8$ & $\Z/4\Z$ & $E(\Q)_{\mathrm{tors}}$  \\
$17a4$ & $\Z/4\Z$ & $E(\Q)_{\mathrm{tors}}$  \\
$19a3$ & $\Z/3\Z$ & $E(\Q)_{\mathrm{tors}}$  \\
$20a1$ & $\Z/6\Z$ & $E(\Q)_{\mathrm{tors}}$  \\
$20a2$ & $\Z/6\Z$ & $E(\Q)_{\mathrm{tors}}$  \\
$21a1$ & $\Z/4\Z \times \Z/2\Z$ & $\{0, (-1, -1), (-2, 1), (5, 8)\}$  \\
$24a1$ & $\Z/4\Z \times \Z/2\Z$ & $E(\Q)_{\mathrm{tors}}$  \\
$24a3$ & $\Z/4\Z$ & $E(\Q)_{\mathrm{tors}}$  \\
$24a4$ & $\Z/4\Z$ & $E(\Q)_{\mathrm{tors}}$  \\
\end{tabular}
\begin{tabular}{c|c|c}
$E$ & $E(\Q)_{\mathrm{tors}}$ & $S_E$ \\
\hline
$26a3$ & $\Z/3\Z$ & $E(\Q)_{\mathrm{tors}}$  \\
$27a3$ & $\Z/3\Z$ & $E(\Q)_{\mathrm{tors}}$  \\
$27a4$ & $\Z/3\Z$ & $E(\Q)_{\mathrm{tors}}$ \\
$30a1$ & $\Z/6\Z$ & $\{0, (3, 4), (-1, 0), (0, -2)\}$  \\
$32a1$ & $\Z/4\Z$ & $E(\Q)_{\mathrm{tors}}$  \\
$32a4$ & $\Z/4\Z$ & $E(\Q)_{\mathrm{tors}}$  \\
$35a3$ & $\Z/3\Z$ & $E(\Q)_{\mathrm{tors}}$  \\
$36a1$ & $\Z/6\Z$ & $E(\Q)_{\mathrm{tors}}$  \\
$36a2$ & $\Z/6\Z$ & $E(\Q)_{\mathrm{tors}}$  \\
$40a3$ & $\Z/4\Z$ & $E(\Q)_{\mathrm{tors}}$  \\
$44a1$ & $\Z/3\Z$ & $E(\Q)_{\mathrm{tors}}$  \\
$54a3$ & $\Z/3\Z$ & $E(\Q)_{\mathrm{tors}}$  \\
$56a1$ & $\Z/4\Z$ & $E(\Q)_{\mathrm{tors}}$  \\
$92a1$ & $\Z/3\Z$ & $E(\Q)_{\mathrm{tors}}$  \\
$108a1$ & $\Z/3\Z$ & $E(\Q)_{\mathrm{tors}}$ 
\end{tabular}
\caption{Some elliptic curves parametrized by modular units}
\end{table}

\begin{remarks}
\begin{enumerate}
\item In order to compute the points $A_j$ in step (3) and $Q_d$ in step (7b), we implicitly make use of Stevens' conjecture that $c_E=1$. This conjecture is known for all elliptic curves of conductor $\leq 200$ \cite{stevens}.
\item Of course, steps (2), (7a) and (7ci) are done only once for each isogeny class.
\item If $x$ is a cusp of $X_1(N)$, then the order of $\varphi_E(x)$ is bounded by the exponent of the cuspidal subgroup of $J_1(N)$. Hence we may ascertain that $\varphi_E(x)$ is rational or not by a finite computation.
\item We compute $e_{\varphi_0}(\frac1d)$ by a numerical method. It would be better to use an exact method.
\end{enumerate}
\end{remarks}

\section{Further questions}\label{questions}

Note that in Lemma \ref{lem preimage}, we considered functions on $E$ which are supported in $E(\Q)_{\mathrm{tors}}$. In general, the image by $\varphi_E$ of a cusp of $X_1(N)$ is only rational over $\Q(\zeta_N)$, and we may use functions on $E$ supported in these non-rational points. In fact, let $S'_E$ denote the set of points $P \in E(\Q(\zeta_N))_{\mathrm{tors}}$ such that $\varphi_E^{-1}(P) \subset C_1(N)$. The set $S'_E$ is stable under the action of $\Gal(\Q(\zeta_N)/\Q)$. Then $E$ can be parametrized by modular units if \emph{and only if} there exist two functions $f,g \in \Q(E)^\times$ supported in $S'_E$ such that $\Q(E)=\Q(f,g)$. As the next example shows, this yields new elliptic curves parametrized by modular units.

\begin{example}\label{ex49}
Consider the elliptic curve $E=X_0(49)=49a1 : y^2+xy=x^3-x^2-2x-1$. The group $E(\Q)_{\mathrm{tors}}$ has order $2$ and is generated by the point $Q=(2,-1)$, which is none other than the cusp $\infty$ (recall that the cusp $0$ is the origin of $E$). The set $S'_E$ consists of all cusps of $X_0(49)$. Let $P$ be the cusp $\frac17$. It is defined over $\Q(\zeta_7)$ and its Galois conjugates are given by $\{P^\sigma\}_\sigma = \{P,3P+Q,-5P,-P+Q,-3P,5P+Q\}$. There exists a function $v \in \Q(E)$ of degree $7$ such that $\dv(v) = \sum (P^\sigma) + (Q) - 7(0)$. Since $x-2$ and $v$ have coprime degrees, the curve $E$ can be parametrized by the modular units $u=x-2$ and $v$.
\end{example}

\begin{example}\label{ex64}
Consider the elliptic curve $E=64a1 : y^2=x^3-4x$. Its rational torsion subgroup is given by $E(\Q)_{\mathrm{tors}} \cong \Z/2\Z \times \Z/2\Z$. There is a degree 2 morphism $\varphi_0 : X_0(64) \to E$, and we have $S_E = E(\Q)_{\mathrm{tors}}$. However, the image of the cusp $\frac18$ is given by $P=\varphi_0(\frac18) = (2i,-2\sqrt{2}+2i\sqrt{2})$. This point is defined over $\Q(\zeta_8)$ and we have $S'_E = S_E \cup \{P^\sigma\}_\sigma$. We can check that $\mathcal{F}_{S'_E}/\Q^\times$ is generated by $x$,  $x \pm 2$ and $x^2+4$, hence it cannot generate $\Q(E)$. However, if we base change to the field $\Q(\sqrt{2})$, then we find that the function $v=y-\sqrt{2} x+2\sqrt{2}$ is supported in $S'_E$ and has degree $3$. Hence $E/\Q(\sqrt{2})$ can be parametrized by the modular units $u=x$ and $v$.
\end{example}

Example \ref{ex64} suggests the following question : which elliptic curves $E/\Q$ of conductor $N$ can be parametrized by modular units \emph{defined over $\Q(\zeta_N)$}? Note that much of the argument in Section \ref{finiteness} is purely geometrical; however, we are crucially using the fact that the modular parametrization is defined over $\Q$.

Finally, here are several questions to which I don't know the answer.

\begin{question}
Let $E/\Q$ be an elliptic curve of condutor $N$. Assume $E$ can be parametrized by modular units of some level $N'$ (not necessarily equal to $N$). Then we have a non-constant morphism $X_1(N') \to E$ and $N$ must divide $N'$. Does it necessarily follow that $E$ admits a parametrization by modular units of level $N$? In other words, does it make a difference if we allow modular units of arbitrary level in Definition \ref{def param modunits}? Similarly, does it make a difference if we replace $Y_1(N)$ by $Y(N)$ or $Y(N')$ in Definition \ref{def param modunits}?
\end{question}

\begin{question}
Does it make a difference if we allow the function field of $E$ to be generated by more than two modular units in Definition \ref{def param modunits}?
\end{question}

\begin{question}
What about elliptic curves over $\C$? It is not hard to show that if $E/\C$ can be parametrized by modular functions, then $E$ must be defined over $\Qb$. In fact, by the proof of Serre's conjecture due to Khare and Wintenberger, it is known that the elliptic curves over $\Qb$ which can be parametrized by modular functions are precisely the $\Q$-curves \cite{ribet}. Which $\Q$-curves can be parametrized by modular units?
\end{question}

\begin{question}
It is conjectured in \cite{bggp} that only finitely many smooth projective curves over $\Q$ of given genus $g \geq 2$ can be parametrized by modular functions. Is it possible to prove, at least, that only finitely many smooth projective curves over $\Q$ of given genus $g \geq 2$ can be parametrized by modular units?
\end{question}

\begin{question}
According to \cite{bggp}, there are exactly 213 curves of genus 2 over $\Q$ which are new and modular, and they can be explicitly listed. Which of them can be parametrized by modular units?
\end{question}

\begin{question}
Let $u$ and $v$ be two multiplicatively independent modular units on $Y_1(N)$. Assume that $u$ and $v$ do not come from modular units of lower level. Can we find a lower bound for the genus of the function field generated by $u$ and $v$?
\end{question}

\bibliographystyle{smfplain}
\bibliography{references}

\end{document}